\documentclass{birkjour}

\usepackage[usenames]{color}
\usepackage{amssymb}
\usepackage{bbm}
\usepackage{graphicx}
\usepackage{amscd}

\usepackage[colorlinks=true,
linkcolor=webgreen,
filecolor=webbrown,
citecolor=webgreen]{hyperref}

\definecolor{webgreen}{rgb}{0,.5,0}
\definecolor{webbrown}{rgb}{.6,0,0}

\usepackage{color}
\usepackage{float}

\usepackage{graphics,amsmath,amssymb}
\usepackage{amsthm}
\usepackage{amsfonts}
\usepackage{latexsym}
\usepackage{epsf}
\usepackage{comment}

\theoremstyle{plain}
\newtheorem{theorem}{Theorem}
\newtheorem{lemma}[theorem]{Lemma}
\newtheorem{corollary}[theorem]{Corollary}

\newtheorem{prop}[theorem]{Proposition}

\theoremstyle{definition}

\makeatother

\begin{document}

\title[Triangulations with few ears]{Triangulations with few ears:  \\ symmetry classes and disjointness}

\author[Andrei Asinowski]{Andrei Asinowski}
\address{%
Institut f\"ur Informatik, \\ Freie Universit\"at Berlin,\\
Takustra\ss e 9, 14195, \\ Berlin, Germany.}
\email{asinowski@mi.fu-berlin.de}
\thanks{Research of the first  author is supported by the ESF EUROCORES programme EuroGIGA, CRP `ComPoSe',
Deutsche Forschungsgemeinschaft (DFG), grant FE 340/9-1.}
\author[Alon Regev]{Alon Regev}
\address{Department of Mathematical Sciences, \\
Northern Illinois University, \\
DeKalb, IL.}
\email{regev@math.niu.edu}


\date{today}

\begin{abstract}
An ear in a triangulation $T$ of a convex $n$-gon $P$
is a triangle of $T$ that shares two sides with $P$ itself.
Certain enumerational and structural problems become easier when one considers only
triangulations with few ears.
We demonstrate this in two ways.
First, for $k=2, 3$, we find the number of symmetry classes of triangulations with $k$ ears.
Second, for $k=2, 3$, we determine the number of triangulations disjoint from a given triangulation:
this number depends only on $n$ for $k=2$,
and only on lengths of branches of the dual tree for $k=3$.
%
%
%
\end{abstract}

\maketitle

\section{Introduction}
Let $P$ be a convex $n$-gon ($n \geq 4$)
with vertices labelled $0, 1, 2, \dots, n-1$,
and let $T$ be its triangulation.
An \textit{ear} is a triangle of $T$ that shares two sides with $P$.
An \textit{internal triangle} is a triangle of $T$ that shares no side with $P$.
The number of ears of any triangulation $P$ is two more than the number of its internal triangles.
The triangulations in Figure~\ref{hexs}(a, b) have $2$ ears,
and the triangulation in Figure~\ref{hexs}(c) has $3$ ears.
More examples of $2$-eared triangulations are given in Figures~\ref{conj-Fig}, \ref{subsets-Fig} and~\ref{Fig1}.

Recall that the number of triangulations of $P$ is $C_{n-2}$,
where $C_n = \frac{1}{n+1} \binom{2n}{n}$ is the $n$th Catalan number.
Hurtado and Noy~\cite{HN} found that the number of triangulations of an $n$-gon with exactly $k$ ears is
\[
\frac{n}{k} 2^{n-2k} \binom{n-4}{2k-4} C_{k-2}
\]
(see also~\cite{C}).
Triangulations having few ears
(equivalently, with few internal triangles)
are generally easier to understand
than triangulations with an arbitrary number of ears.
Adin, Firer and Roichman \cite{AFR} showed that the affine Weyl group $\tilde{C}_n$ acts transitively,
by ``colored'' flips, on the set of $2$-eared triangulations.

The purpose of this note is to study two aspects of $2$-eared and $3$-eared triangulations:
we enumerate their symmetry classes, and we compute the number of triangulations which share no diagonals with a given $2$-eared or $3$-eared triangulation.
Extension of these results to general triangulations is interesting but more difficult.

\section{Symmetry classes}
The action of the dihedral group of order $2n$ on the triangulations of an $n$-gon by rotation and reflection\footnote{To this end, we consider $P$ as a regular polygon.}
defines an equivalence relation, whose classes are called {\em symmetry classes}.
Figure~\ref{hexs} shows representatives of all three symmetry classes of triangulations of a hexagon.
Moon and Moser \cite{MM} enumerated the symmetry classes of triangulations of an $n$-gon.
Clearly, rotations and reflections
preserve the number of ears, so that the equivalence relation may be restricted to triangulations with a fixed number of ears.

\begin{figure}[h]
\begin{center}
\includegraphics[scale=0.6]{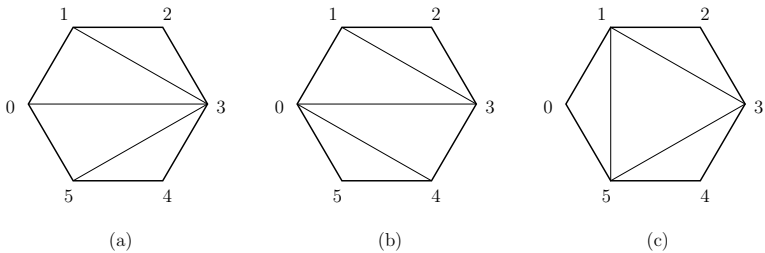}
\end{center}
\caption{Triangulations of a hexagon.}\label{hexs}
\end{figure}


To enumerate the symmetry classes of $2$-eared triangulations, we note their connection with integer compositions.
The two ears of a $2$-eared triangulation of an $n$-gon
partition the remaining $n-4$ sides of the $n$-gon
into two paths.
Considering a drawing in which the ears are separated by a vertical line, we can call these paths ``up'' and ``down''.
See Figure~\ref{conj-Fig}: the ``up path'' is $123456$, and the ``down path'' is $89(10)$.
Each of the remaining $n-4$ triangles of the triangulation ``points'' either up or down.
For example, in the triangulation in Figure~\ref{conj-Fig},
the triangles with vertices $2, 9, 10$ and $6, 8, 9$ are the ones pointing up.
It is easily seen that every configuration of triangles pointing up and down is possible,
and that together with the position of the ears this configuration completely determines the $2$-eared triangulation.
The connection with integer compositions is apparent if we associate with each composition of $n-3$ a sequence of bars placed in the $n-4$ spaces between $n-3$ balls.
Then each bar corresponds to an upwards-pointing triangle and each vacant space between balls corresponds to a downwards-pointing triangles.
Every symmetry class can be represented in this way four times (some of the representations might coincide),
see Figure~\ref{conj-Fig}.
The operation of interchanging the bars and the vacant spaces is known as {\em conjugation} of compositions,
first studied by MacMahon~\cite{Ma}.
In terms of triangulations, conjugation corresponds to interchanging the up and down directions,
and composition reversal (i.e., taking $x_1+x_2+ \ldots +x_n$ to $x_n+x_{n-1}+\ldots +x_1$) corresponds to reflection about a vertical axis.
This leads to the following conclusion.

\begin{prop}\label{bij}
There is a bijection between symmetry classes of $2$-eared triangulations of an $n$-gon and equivalence classes of compositions of $n-3$,
where two compositions are considered equivalent if one can be obtained from the other by conjugation and reversal.
\end{prop}


Now we are able to enumerate the symmetry classes for $2$-eared triangulations.

\begin{theorem}
There are $2^{n-6}+2^{\lfloor n/2 \rfloor-3 }$ symmetry classes of $2$-eared triangulations of an $n$-gon.
\begin{proof}
By Proposition \ref{bij} it suffices to enumerate the conjugacy-and-reversal classes of compositions.
We do this by counting the compositions invariant under each relevant operation.
In all, there are $2^{m-1}$ compositions of an integer $m$. By results of MacMahon \cite{Ma,Mu},
there are $2^{\lfloor m/2 \rfloor}$ compositions of $m$ which are equal to their own reversal,
and if $m$ is odd, there are $2^{\lfloor m/2 \rfloor}$ compositions of $m$ whose conjugate equals their reversal.
If $m$ is even, the conjugate of a composition of $m$ cannot be equal to its reversal.
Finally, for $m>1$ a composition cannot be self-conjugate.

\begin{figure}[H]
\begin{center}
\includegraphics[scale=0.55]{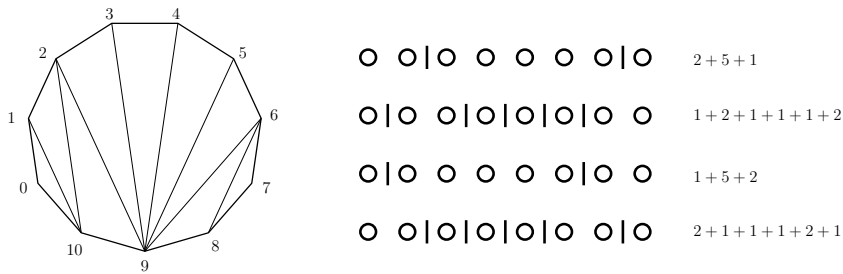}
\end{center}
\caption{A $2$-eared
 triangulation of a $11$-gon and corresponding compositions of $8$.}\label{conj-Fig}
\end{figure}

Thus by the Cauchy-Frobenius Lemma~\cite{B}
(applied to the group of order $4$ generated by the reversal and conjugation operations), the number of equivalence classes of compositions of $m>1$ is
\[
\frac{1}{4} (2^{m-1}+2^{\lfloor m/2 \rfloor}+{\mathbbm 1}_{2\nmid  m} \, 2^{\lfloor m/2 \rfloor})=2^{m-3}+2^{\lfloor {(m-3)/2 \rfloor}},
\]
where the indicator ${\mathbbm 1}_{2\nmid  m}$ is $1$ if $m$ is odd and $0$ if $m$  even. Setting $m=n-3$ then gives the stated result.
\end{proof}
\end{theorem}

The sequence enumerating these symmetry classes is given in 
The On-Line Encyclopedia of Integer Sequences~\cite[A005418]{oeis}.
An analysis along the same lines can be performed to enumerate the symmetry classes of triangulations with $k$ ears.
For example, the number of symmetry classes of $3$-eared triangulations of an $n$-gon can be determined
by counting those that are invariant under rotation by $120^{\circ}$ and those that are invariant under reflection.
The result is that there are
\[
\frac{1}{3}  2^{n-8} (n-4)(n-5) + {\mathbbm 1}_{2 | n} 2^{n/2-4}+{\mathbbm 1}_{3 |  n} \frac{1}{3} 2^{n/3-2}\\
\]
symmetry classes of $3$-eared triangulations. Similar computations for larger $k$ are possible but become rather tedious.

\section{Disjoint triangulations}
\subsection{The two-eared case}

We call two triangulations {\em disjoint} if they have no diagonals in common.
Huguet and Tamari~\cite{HT} gave a complicated recursive formula for the total number of ``proper diagonals'' of the associahedron, or equivalently, the total number of pairs of disjoint triangulations of an $n$-gon. 
Such disjoint pairs are also related to questions about distances in the associahedron,
since, by a lemma of Sleator, Tarjan and Thurston \cite[Lemma 3]{STT},
minimal paths on the associahedron always avoid them when possible.
Given a triangulation, it may be difficult to determine how many triangulations are disjoint from it.
However, for $2$-eared triangulations the answer is simple.
\begin{theorem} \label{same}
If $T$ is any $2$-eared triangulation of an $n$-gon,
there are $C_{n-3}$ triangulations disjoint from $T$.
\begin{proof}
Consider the $2$-eared triangulation with diagonals $13, 14, \ldots, 1(n-1)$
(see the left side of Figure \ref{subsets-Fig}).
This is known as an {\em arrow} triangulation.
It is easily seen that a triangulation $T'$ is disjoint from this arrow triangulation if and only if
$T'$ includes the diagonal $02$.
Therefore there are $C_{n-3}$ triangulations disjoint from this arrow triangulation.

Now let $T_1$ and $T_2$ be two $2$-eared triangulations,
and we show that
the number of triangulations disjoint from $T_1$
is equal to
the number of triangulations disjoint from $T_2$.
Note that both $T_1$ and $T_2$ have linear dual trees, so we can define an order $<$ on their
diagonals.
 This induces a bijection  $\Phi$ between the diagonals of $T_1$ and of $T_2$. For example, for the triangulations of Figure  \ref{subsets-Fig}, we may assign
$\Phi(13)=68$, $\Phi(14)=69$, $\Phi(15)=59$,
and so on. If $e_1 <e_2< \ldots < e_r$  are some
diagonals
of $T_1$, we claim that the number of triangulations which include the
diagonals
$e_1,\ldots, e_r$ is the same as the number of triangulations that include the
diagonals
$\Phi(e_1),\ldots , \Phi(e_r)$. Indeed, since the
diagonals
$e_i$ and the
diagonals
$\Phi(e_i)$ have the same relative positions in their respective triangulations, they partition their $n$-gon into the same sized polygons. The number of triangulations including at least these diagonals is then just the number of ways to triangulate these polygons, which is the same number for $T_1$ and $T_2$.
For example, for each of the triangulations given in Figure \ref{subsets-Fig}, there are $C_2 C_2 C_4 C_1$ triangulations that include the marked diagonals.
Thus, the number of triangulations sharing each subset of the diagonals of $T_1$ is the same as the number of triangulations sharing the corresponding subset of $T_2$.
In particular this is true for the empty set,
and the theorem follows from considering any $2$-eared triangulation and an arrow triangulation.
\end{proof}
\end{theorem}

\begin{figure}[h]
\begin{center}
\includegraphics[scale=0.6]{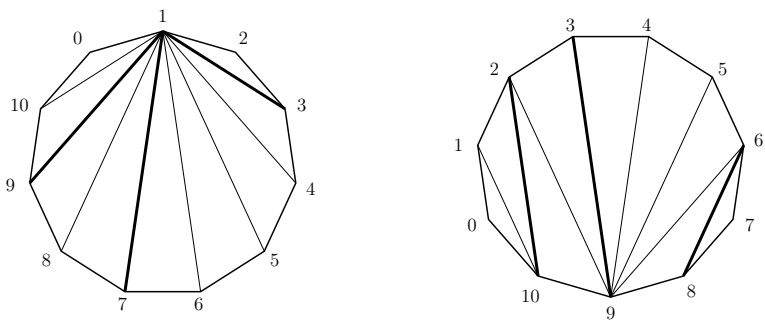}
\end{center}
\caption{Two $2$-eared
 triangulations with corresponding subsets of diagonals.}\label{subsets-Fig}
\end{figure}


Next we make some remarks related to Theorem~\ref{same} and its proof.

\begin{enumerate}
\item The proof of Theorem~\ref{same} can be also presented as follows.
The number of triangulations of an $n$-gon \textit{containing} certain fixed diagonals
is the product of the numbers of triangulations of sub-polygons created by these diagonals,
that is -- the product of corresponding Catalan numbers.
Therefore, the number $\mathrm{disj}(T)$ of triangulations \textit{avoiding} all the diagonals of a given triangulation $T$
is given by an inclusion-exclusion formula.
If $T$ is a $2$-eared triangulation of an $n$-gon,
then, due to the linear ordering of its diagonals,
the terms in this formula will be products of Catalan numbers
$C_{a_1}C_{a_2}\dots C_{a_i}$, where $(a_1, a_2, \dots, a_i)$ are all possible compositions of $n-2$.
Namely,
\begin{equation}\label{eq:cat}
\mathrm{disj}(T) = \sum_{i=1}^{n-2} \left( (-1)^{i+1} \prod_{
\begin{array}{c}
(a_1, a_2, \dots, a_i): \\
a_1, a_2, \dots, a_i \in \mathbb{N}, \\
a_1+ a_2+ \dots+ a_i = n-2
\end{array}
} C_{a_i}
\right).
\end{equation}


This already proves that for given $n$, $\mathrm{disj}(T)$ does not depend on a specific $T$.
In order to find analytically $\mathrm{disj}(T)$ 
we notice that
$\mathrm{disj}(T)$, as given in~\eqref{eq:cat},
is the coefficient of $x^{n-3}$ in
\[\sum_{i\geq 0} (-1)^{i} x^i s^{i+1}(x),\]
where $c(x)$ is the generating function of Catalan numbers, and $s(x)=\frac{c(x)-1}{x} = c^2(x)$.
The sum of this geometric series is
\[
\frac{s(x)}{1+xs(x)} = \frac{c(x)-1}{xc(x) } = c(x),
\]
and, therefore, $\mathrm{disj}(T)$
is indeed equal to $C_{n-3}$.

\item
We note another interpretation of the result of Theorem \ref{same}.
The second author \cite{R1} showed that the number of triangulations of a $2n$-gon
without any diagonals parallel to one of its sides (for example, to $01$),
is equal to 
$2C_{2n-3}$.
Similarly, consider the triangulations of an $n$-gon which avoid diagonals parallel to either $01$ or $02$.
These are precisely the triangulations that are disjoint from the $2$-eared
 triangulation with
 diagonals
 $02, 2(n-1), (n-1)3, 3(n-2), \ldots ,(\lfloor n/2 \rfloor+2)(\lfloor n/2 \rfloor)$,
 known as a {\em snake} triangulation (see Figure~\ref{Fig1}).
Thus we obtain the following corollary.
\begin{corollary}
For any $n\ge 3$, the number of triangulations of an $n$-gon avoiding any diagonals parallel to $01$ or $02$, is equal to $C_{n-3}$.
\end{corollary}

\begin{figure}[h]
\begin{center}
\includegraphics[scale=0.25]{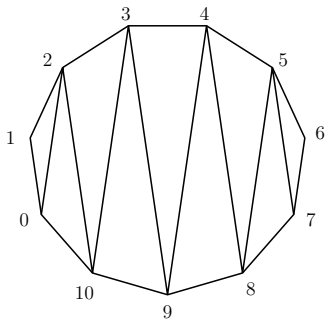}
\end{center}
\caption{A snake triangulation of an $11$-gon.}\label{Fig1}
\end{figure}

\item
The idea of the proof of Theorem~\ref{same} can be extended to proving the following more general result.
\begin{theorem}\label{same-gen}
The number of triangulations disjoint from a triangulation $T$ of an $n$-gon depends only on the position of the internal triangles of $T$.
\end{theorem}


The proof is essentially the same:
consider two triangulations, $T_1$ and $T_2$, with their internal triangles in the same position. 
There is a natural bijection between the diagonals of $T_1$ and $T_2$
such that corresponding diagonals partition $P$ into the same size polygons.
Therefore, again,
the number of triangulations sharing a subset of the diagonals of $T_1$ is equal to
the number of triangulations sharing the corresponding subset of $T_2$;
and, in particular, this holds for the empty set.
In other words, computing the number of triangulations disjoint from $T_1$
and those disjoint from $T_2$ will produce the same inclusion-exclusion formula.

\end{enumerate}

\subsection{The three-eared and general case}

In what follows we apply the idea of Theorem \ref{same-gen} to the $3$-eared case. We start with the following lemma.

\begin{lemma}\label{thm:m_forb}
Let $a$ be a fixed vertex of $P$.
Let $m$ be a number such that $0 \leq m \leq  n-3$.
The number of triangulations of $P$
that do not use any of the $m$ diagonals 
$a(a+2), a(a+3), \dots, a(a+m+1)$, is
\begin{equation}\label{eq:m_forb}
\sum_{i=0}^{n-3-m} C_{i} C_{n-3-i}.
\end{equation}

\end{lemma}

\begin{proof}
If $m=0$, there is no restriction, and we have $C_{n-2}$ triangulations, which is equal to
$\displaystyle\sum_{i=0}^{n-3} C_{i} C_{n-3-i}$.

Now suppose $1 \leq m \leq  n-3$.
We assume without loss of generality that $a=0$, refer to Figure~\ref{fig:m_forb}.
A triangulation $T'$ that satisfies the condition cannot use the diagonal $02$.
Therefore, $T'$ has at least one diagonal with an endpoint in $1$.
Let $b$, $3 \leq b \leq n-1$, be the \textbf{maximal} number so that $1b$ is a diagonal in $T'$.
Then $T'$ necessarily has the diagonal $0 b$ (unless $b=n-1$).
Therefore we have $b > m+1$ because otherwise $0 b$ is forbidden. 
\begin{figure}[h]
$$\includegraphics[width=45mm]{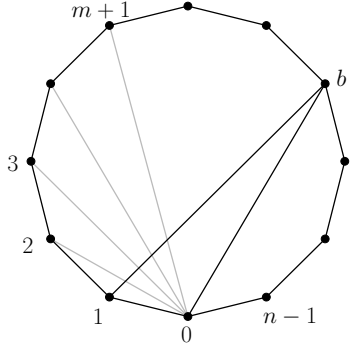}$$
\caption{Illustration to proof of Lemma~\ref{thm:m_forb}. The forbidden diagonals are shown by grey color.}
\label{fig:m_forb}
\end{figure}

For any choice of triangulation of the $b$-gon with vertices  $1, 2, \dots, b$
and of the $(n-b+1)$-gon with vertices  $0, b, b+1, \dots, n-1$,
a valid triangulation of $X_n$ is obtained.
Since $m+2 \leq b \leq n-1$,
the number of triangulations $T'$ that satisfy the condition is
\[
C_{m}C_{n-m-3}+C_{m+1}C_{n-m-4}+ \dots +C_{n-3}C_{0}=
\sum_{i=0}^{n-3-m} C_{i} C_{n-3-i}.\]
\end{proof}


Let $T$ be a $3$-eared triangulation of an $n$-gon.
We say that $T$ has type $(p, q, r)$
if $p, q, r$ are the numbers of triangles in three parts of $T$ separated by its (unique) internal triangle.
In other words, the dual graph of such $T$ consists of three branches, of lengths $p$, $q$ and $r$,
connected to the common point.
Notice that $p+q+r=n-3$.
The $3$-eared triangulation in Figure~\ref{hexs}(c) has type $(1,1,1)$.

Now we prove the main result of this section.

\begin{prop}\label{thm:pqr}
Let $T$ be a $3$-eared triangulation
of type $(p,q,r)$.
The number of triangulations of $X_n$
disjoint from $T$ is
\begin{equation}\label{eq:pqr}
2C_{n-3}- \sum_{i=0}^p C_i C_{n-4-i}- \sum_{i=0}^q C_i C_{n-4-i}- \sum_{i=0}^r C_i C_{n-4-i}.
\end{equation}
\end{prop}

\begin{proof}
By 
Theorem \ref{same-gen} it suffices to prove the result for
triangulation $T$ that appears in Figure~\ref{fig:pqr}.
\begin{figure}[h]
$$\includegraphics[width=60mm]{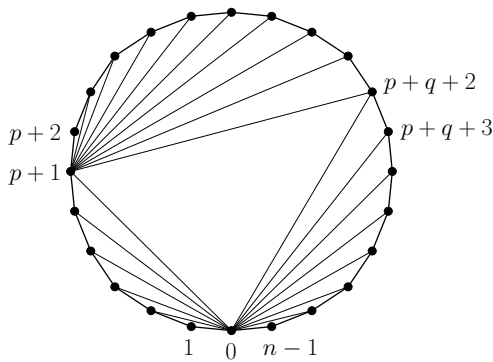}$$
\caption{The representative for triangulations of type $(p, q, r)$.}
\label{fig:pqr}
\end{figure}

Let $T'$ be a triangulation disjoint from $T$.
We split the discussion into two cases, conditioning on whether or not $1 (n-1)$ belongs to $T'$.

\begin{itemize}
\item Case $1$: $1 (n-1)$ belongs to $T'$.

$T'$ is disjoint form $T$
if and only if
the 
triangulation of the $(n-1)$-gon with vertices  $1, 2, \dots, n-1$
avoids the diagonals
$(p+1)(p+3), (p+1)(p+4), \dots, (p+1)(p+q+2)$.
By Lemma~\ref{thm:m_forb}, the number of such triangulations is given by~\eqref{eq:m_forb}
(with $n-1$ in role of $n$, and $q$ in role of $m$),
that is,
\begin{equation}\label{eq:pqr1}
\sum_{i=0}^{n-4-q} C_{i} C_{n-4-i}.
\end{equation}

\item Case $2$: $1(n-1)$ does not belong to $T'$.

In such a case $T'$ has a diagonal one of whose endpoints is $0$.
Denote by $d$ the \textbf{maximal} number so that $0d$ is a diagonal in $T'$.
Obviously, $p+2 \leq d \leq p+q+1$.
Notice that $(n-1)d$ is also a diagonal in $T'$.
Now we have a situation shown in Figure~\ref{fig:pqr2}.
\begin{figure}[h]
$$\includegraphics[width=60mm]{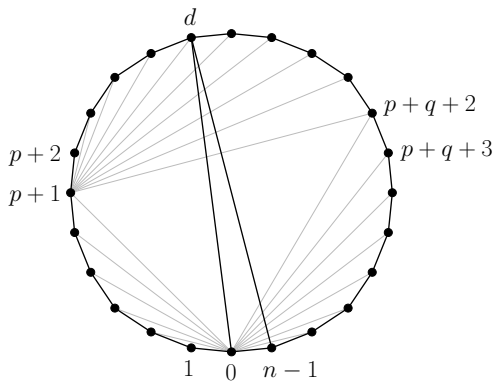}$$
\caption{Illustration to proof of Proposition~\ref{thm:pqr}, case $2$.}
\label{fig:pqr2}
\end{figure}

$T'$ is disjoint from $T$ if and only if
the 
triangulation ($T_1$) of the $(d+1)$-gon with vertices $0, 1, \dots, d$
avoids diagonals $02, 03, \dots, 0(p+1)$
and $(p+1)(p+3), (p+1)(p+4), \dots, (p+1)d$.
There are no restrictions on the 
triangulation ($T_2$) of the $(n-d)$-gon with vertices $d, d+1, \dots, n-1$.
For the the triangulation $T_1$, the condition is equivalent to being disjoint from
a triangulation with two ears -- $0 1 2$ and $(p+1)(p+2)(p+3)$.
Therefore, by Theorem~\ref{same}, the number of such triangulations $T_1$ is
$C_{d-2}$.
The number of triangulations $T_2$ is
$C_{n-d-2}$.
Since $p+2 \leq d \leq p+q+1$,
the number of triangulations $T'$ disjoint from $T$
is
\begin{equation}\label{eq:pqr2}
C_{p}C_{n-p-4} + C_{p+1}C_{n-p-3} + \dots + C_{p+q-1}C_{n-p-q-3}=
\sum_{j=p}^{p+q-1}C_{j}C_{n-4-j}.
\end{equation}
\end{itemize}
Summing up~\eqref{eq:pqr1} and~\eqref{eq:pqr2}, we obtain
\begin{equation*}
\sum_{i=0}^{n-4-q} C_{i} C_{n-4-i} \ + \ \sum_{j=p}^{p+q-1}C_{j}C_{n-4-j}.
\end{equation*}
This can be shown to be equivalent to~\eqref{eq:pqr},
by using the standard convolution formula $\sum_{i=0}^{n-4} C_{i} C_{n-4-i}=C_{n-3} $.
\end{proof}


\textit{Remark.} The result of Theorem~\ref{same}
can be seen as a special case of Theorem~\ref{thm:pqr}
with $r=0$, $q=n-3-p$.

\section{Concluding remarks}

It would be interesting to find a formula for the number of triangulations disjoint from a fixed $k$-eared triangulation in terms of the position of its internal triangles.
As above, one can take a convenient triangulation with the given inner triangles.
Such results may be used to obtain another formula for the total number of pairs of disjoint triangulations of an $n$-gon.

\end{document}